\documentclass{amsart}
\usepackage{amssymb,amsmath,amsfonts,amsthm}
\usepackage{graphicx}
\graphicspath{ {./images/} }

\usepackage{mathrsfs}  
\usepackage{psfrag}
\usepackage{mathtools}
\usepackage{color}
\usepackage{todonotes}
\usepackage{enumitem}

\usepackage{chngcntr}
\counterwithin{equation}{section}

\theoremstyle{plain}
\newtheorem{main}{Theorem}

\newtheorem{theorem}{Theorem}[section]
\newtheorem{lemma}[theorem]{Lemma}
\newtheorem{proposition}[theorem]{Proposition}

\theoremstyle{remark}
\newtheorem{remark}[theorem]{Remark}
\newtheorem{definition}{Definition}

\newcommand\numberthis{\addtocounter{equation}{1}\tag{\theequation}}

\newcommand{\diam}{\operatorname{diam}}

           \def\ea{\end{array}}
          \def\ec{\end{center}}
     \def\ed{\end{description}}
        \def\ee{\end{equation}}
       \def\eea{\end{eqnarray}}
     \def\eeaa{\end{eqnarray*}}
 \def\et{\end{thebibliography}}

\def\bM{{\bf{M}}}

\def\cG{{\mathcal G}}
\def\cA{{\mathcal A}}
\def\cD{{\mathcal D}}
\def\cC{{\mathcal C}}

\def\cM{{\mathcal M}}

\def\cP{{\mathcal P}}

\def\cS{{\mathcal S}}

\def\vep{\varepsilon}

\def\TT{{\mathbb T}}
\def\RR{{\mathbb R}}

\def\NN{{\mathbb N}}

\def\Var{\operatorname{Var}}
\def\Exp{\operatorname{Exp}}

\def\vp{{\varphi}}

\title[Existence and uniqueness of equilibrium states]{Existence and uniqueness of equilibrium states for systems with specification at a fixed scale}
\date{\today}
\author{Maria Jose Pacifico, Fan Yang and Jiagang Yang}

\address{Instituto de Matem\'atica, Universidade Federal do Rio de Janeiro, C. P. 68.530, CEP 21.945-970,  Rio de Janeiro, RJ, Brazil.}
 \email{pacifico@im.ufrj.br }

\address{Department of Mathematics, Michigan State University, East Lansing, Michigan, USA.}
\email{yangfa31@msu.edu}

\address{Departamento de Geometria, Instituto de Matem\'atica e Estat\'\i stica, Universidade Federal Fluminense, Niter\'oi, Brazil}
\email{yangjg\@@impa.br}

\thanks{This research has been supported [in part] by CAPES - Finance Code $001$ and CNPq-grants. MJP was partially supported by FAPERJ}

\begin{document}

\maketitle

\begin{abstract}
We consider the uniqueness of equilibrium states for dynamical systems that satisfy certain weak, non-uniform versions of specification, expansivity, and the Bowen property at a fixed scale. Following Climenhaga-Thompson's approach which was originally due to Bowen and Franco, we prove that equilibrium states are unique even when the weak specification assumption only holds on a small collection of orbit segments.
\end{abstract}

\section{Introduction}\label{s.1}
For a continuous flow $(\varphi_t)_{t\in\RR}$ or a homeomorphism $f$ on a compact metric space $\bM$ and a continuous function $\phi:\bM\to\RR$, an invariant probability measure that maximize the quantity
$$
P(\mu):=h_\mu(\varphi_1) + \int\phi\,d\mu
$$
is called an {\em equilibrium state} for the potential function $\phi$. Such measures include the measures of maximal entropy (for $\phi\equiv 0$) and the physical measures of Axiom A attractors (when $\phi=-\log |\det Df\mid E^u|$). The study of such measures traces back to Sinai, Ruelle and Bowen~\cite{B08,Ru,Si} and has a deep connection with statistical physics.

This paper is devoted to the existence and, more importantly, the uniqueness of equilibrium states. This problem has been studied by various authors using topological (e.g.~\cite{B75,Franco}) and functional analytical (e.g.~\cite{RuBook,Sarig}) approaches. Among them, a recent breakthrough was made by Climenhaga and Thompson~\cite{CT16} following Bowen's approach~\cite{B75}. They provide a easy-to-verify topological criterion consisting of the following ingredients:
\begin{enumerate}[label={(\Roman*)}]
\item[(0)] there exists a large collection of orbit segments $\cD$ with a `decomposition' $(\cP,\cG,\cS)$; thinking of $\cG$ as the `good core', we have
$$
\cG\subset \cG^1\subset\cG^2\subset\cdots\subset \cD
$$
which eventually exhausts the set $\cD$;
\item the system has the specification property on $\cG^K$ at scale $\delta$, for every $K\in\NN$;
\item the potential function $\phi$ has bounded distortion (the Bowen property) on $\cG$ at a given scale $\vep$;
\item the `bad' parts of the system, consisting of $\cP,\cS,\cD^c$ and those points where the system is not expansive, must have smaller pressure comparing to $\cG$.
\end{enumerate}
Under these assumptions, they prove that there exists a unique equilibrium state with the upper and lower Gibbs property. For an overview of their result and applications, see the recent survey~\cite{CT19} and the references therein.

It is worth noting that all the assumptions listed above are made at certain prescribed scales. This is particularly useful for certain applications such as the Bonatti-Viana diffeomorphism on $\TT^4$~\cite{CFT18}, Ma\~n\'e's derived from Anosov diffeomorphism on $\TT^3$~\cite{CFT19}, and geodesic flows on surfaces without conjugate points~\cite{CKW}. In many other applications however, Assumption (I) is often replaced by the following, much stronger assumption (see Lemma~\ref{l.tailspec}):
\begin{enumerate}[label={(\Roman*)}]
	\item[(I')] The system has the specification property on $\cG$ at {\bf all scales}.
\end{enumerate}

The goal of this paper is to improve the Climenhaga-Thompson criterion by weakening the specification assumption. For this purpose, we assume $(\varphi_t)_{t\in\RR}$ to be a Lipschitz flow with Lipschitz constants continuous in $t$. Our result can be easily adapted to Lipschitz continuous homeomorphisms with minor modifications.

We define
\begin{equation}\label{e.L}
L_X = \max_{t\in[0,1]} L_{\varphi_t}\ge 1.\footnote{When $f$ is a Lipschitz homeomorphism, we can take $L_f$ to be the Lipschitz constant of $f$.} 
\end{equation}
Our main result is:
\begin{main}\label{m.1}
Let $(\varphi_t)_{t\in\RR}$ be a Lipschitz continuous flow on a compact metric space $\bM$, and $\phi:\bM\to\RR$ a continuous potential function. Suppose that there exist $\vep>0,\delta>0$ with $\vep\ge1000L_X\delta$ such that $\varphi_t$ is almost expansive at scale $\vep$, and $\cD\subset\bM\times\RR^+$ which admits a decomposition $(\cP,\cG,\cS)$ with the following properties:
\begin{enumerate}[label={(\Roman*)}]
\item $\cG$ has tail (W)-specification at scale $\delta$;
\item  $\phi$ has the Bowen property at scale $\vep$ on $\cG$;
\item $P(\cD^c\cup [\cP]\cup[\cS], \phi,\delta,\vep)<P(\phi)$.
\end{enumerate}
Then there exists a unique equilibrium state for the potential $\phi$.
\end{main}

The terminologies will be explained in detail in the next section.

Comparing to~\cite[Theorem 2.9]{CT16}, here we only assume that the specification property (I) holds on the core collection $\cG$ at a fixed scale. This assumption is  easier to verify for systems where the specification property is known to fail for all scales, such as flows with singularities. Theorem~\ref{m.1} will play a central role in our next paper where we prove the uniqueness of equilibrium states for Lorenz attractors in any dimension~\cite{PYY22}. It can also be applied to nearby systems under small perturbation. Applications along this direction is current underway.

We also remark that the assumptions above are not optimal. In particular, we expect that the same result holds if one replaces the almost expansivity by $P^\perp_{\exp}(\phi,\vep)<P(\phi)$ (for the meaning of $P^\perp_{\exp}(\phi,\vep)$, see~\cite[Definition 2.7, 2.8]{CT16}). The ratio $\vep/\delta\ge1000L_X$ can also be weakened. However, to simplify the argument and the choice of parameters, we will not pursue the optimal ratio in this paper.

\section{Preliminaries}\label{s.2}
For completeness, in this section we will recap the terminology of~\cite{CT16}. For the convenience of those who are familiar with~\cite{CT16}, our notations are largely the same, so it is safe to skip this section and move on to Section~\ref{s.2'}.

\subsection{Pressure on a collection of orbit segments}\label{s.2.1}
Throughout this article, we will assume that $\bM$ is a compact metric space, and $X$ is a Lipschitz vector field on $\bM$. Denote by $\varphi_t:\bM\to \bM$ the continuous flow generated by $X$, and by $\cM_X(\bM)$ the set of Borel probability measures on $\bM$ that are invariant under $X$. Given $x\in \bM$ and real numbers $a<b$, we write 
$$
\varphi_{[a,b]}(x) := \{\varphi_s(x):a\le s\le b\}
$$
for the orbit segment from $\phi_a(x)$ to $\phi_b(x)$. $\varphi_{(a,b)}(x)$ can be defined in a similar way.

For $t>0$, the {\em Bowen metric} is defined as 
$$
d_t(x,y) = \sup\{d(\vp_s(x),\vp_s(y)): 0\le s\le t\}.
$$
For $\delta>0$, the $(t,\delta)$-Bowen ball at $x$ is the $\delta$-ball under the Bowen metric $d_t$
$$
B_t(x,\delta) = \{y\in \bM: d_t(x,y)<\delta\}, 
$$
and its closure
$$
\overline{B}_t(x,\delta) = \{y\in \bM: d_t(x,y)\le\delta\}.
$$

Given $t>0$ and $\delta>0$, a set $E\subset \bM$ is call {\em $(t,\delta)$-separated}, if for every distinct $x,y\in E$, one has $d_t(x,y)>\delta$.

A core concept in the work of Climenhaga and Thompson is the pressure on a collection of orbit segments. Writing $\RR^+ = [0,\infty)$, we regard $\bM\times \RR^+$ as the collection of finite orbit segments by identifying each $(x,t)\in \bM\times \RR^+$ with the orbit segment $\varphi_{[0,t)}(x)$, with the case of $t=0$ being associated with the empty set rather than the singleton $\{x\}$ itself. Given $\cC\subset \bM\times \RR^+$ and $t\ge 0$ we write $\cC_t = \{x\in \bM: (x,t)\in \cC\}$.

Given a continuous function $\phi:\bM\to\mathbb{R}$ which will be called a {\em potential function} and a scale $\vep>0$, we write 
\begin{equation}\label{e.Phi}
\Phi_\vep(x,t) = \sup_{y\in B_t(x,\vep)}\int_0^t\phi(\varphi_s(y))\,ds,
\end{equation}
with $\vep = 0$ being the standard Birkhoff integral
$$
\Phi_0(x,t) =\int_0^t\phi(\varphi_s(y))\,ds.
$$
Putting $\Var(\phi,\vep) = \sup\{|\phi(x)-\phi(y)|: d(x,y)<\vep\}$, we obtain the trivial bound
$$
|\Phi_\vep(x,t) - \Phi_0(x,t)|\le t\Var(\phi,\vep).
$$

The {\em (two-scale) partition function} $\Lambda(\cC,\phi,\delta,\vep,t)$ for $\cC\in \bM\times \RR^+, \delta>0,\vep>0, t>0$ is defined as 
\begin{equation}\label{e.Lambda}
\Lambda(\cC,\phi,\delta,\vep,t)=\sup\left\{\sum_{x\in E}e^{\Phi_\vep(x,t)}:E\subset \cC_t \mbox{ is $(t,\delta)$-separated}\right\}.
\end{equation}
Henceforth we will often suppress the potential function $\phi$ and write $\Lambda(\cC,\delta,\vep,t)$ since the potential will be fixed throughout. 
When $\cC = \bM\times \RR^+$ we will also write $\Lambda(\bM,\delta,\vep,t)$. A $(t,\delta)$-separated set $E\in\cC_t$ achieving the supremum in~\eqref{e.Lambda} is called {\em maximizing} for $\Lambda(\cC,\delta,\vep,t)$. Note that the existence of such set is only guaranteed when $\cC_t$ is compact. Also note that $\Lambda$ is monotonic in both $\delta$ and $\vep$, albeit in different directions: if $\delta_1<\delta_2$, $\vep_1<\vep_2$ then
\begin{equation}\label{e.monotone}
\begin{split}
\Lambda(\cC,\delta_1,\vep,t)\ge\Lambda(\cC,\delta_2,\vep,t),\\
\Lambda(\cC,\delta,\vep_1,t)\le\Lambda(\cC,\delta,\vep_2,t).
\end{split}
\end{equation}

The pressure of $\phi$ on $\cC$ with scale $\delta,\vep$ is defined as
\begin{equation}\label{e.P1}
P(\cC,\phi,\delta,\vep) = \limsup_{t\to+\infty}\frac1t \log\Lambda(\cC,\phi,\delta,\vep,t).
\end{equation}
The monotonicity of $\Lambda$ can be naturally translated to $P$. Note that when $\vep=0$, $\Lambda(\cC,\phi,\delta,0,t)$ and $P(\cC,\phi,\delta,0)$ agree with the classical definition. In this case we will often write $P(\cC,\phi,\delta)$, and let
\begin{equation}\label{e.P2}
P(\cC,\phi) = \lim_{\delta\to0} P(\cC,\phi,\delta). 
\end{equation}
When $\cC = \bM\times \RR^+$, this coincides with the standard definition of the topological pressure $P(\phi)$.

The variational principle for flows~\cite{BR} states that 
$$
P(\phi) = \sup_{\mu\in\cM_X(\bM)}\left\{h_\mu(\varphi_1)+\int\phi\,d\mu\right\},
$$
where $h_\mu(\varphi_1)$ is the metric entropy of the time-one map $\varphi_1$. A measure achieving the supremum, when exists,\footnote{In particular, such measures exist when the metric entropy is upper semi-continuous. See Section~\ref{s.2.2}.} is called an equilibrium state for the potential $\phi$. When $\phi\equiv0 $ we have $P(\phi)=h_{top}(\varphi_1)$~\cite{B71, BR}, and the corresponding equilibrium states are called the measures of maximal entropy.


\subsection{Almost expansivity and $h$-expansivity}\label{s.2.2}
For $x\in \bM$ and $\vep>0$, we consider the set 
$$
\Gamma_\vep(x)=\{y\in \bM: d(\varphi_tx,\varphi_ty)\le \vep \mbox{ for all }t\in\RR\},
$$
which  will be called an (two-sided) infinite Bowen ball at $x$. It is easy to see that $\Gamma_\vep(x)$ is compact for all $x$ and $\vep$.

There are multiple ways to define the expansivity for flows, depending on whether one allows for reparametrization of the time $t$. In this paper we will use a weaker notion of expansivity that does not involve any reparametrization of time. This definition is particularly useful in the entropy theory  of flows.

\begin{definition}\label{d.alexp}
The flow $\varphi_t$ is said to be almost expansive at scale $\vep>0$, if the set 
\begin{equation}\label{e.Exp}
\Exp_\vep(X) :=  \left\{x\in \bM: \Gamma_\vep(x)\subset \varphi_{[-s,s]}(x) \mbox{ for some }s>0\right\}
\end{equation}
has full probability: for any ergodic  $\mu\in\cM_X(\bM)$ one has $\mu(\Exp_\vep(X))=1$.
\end{definition}
By the ergodic decomposition theorem, it follows that $\mu(\Exp_\vep(X))=1$ for any $\mu\in\cM_X(\bM)$ that is not necessarily ergodic. Almost expansivity can be defined for homeomorphisms is a similar way. 

It is clear that if $\varphi_1$ is expansive, then the flow is almost expansive (at a slightly different scale). 
The motivation for the present paper is the fact that sectional-hyperbolic flows are almost expansive. See~\cite{PYY}.

Below we will use Bowen's definition~\cite{B71} of the topological entropy (see also~\cite{Wal}).
\begin{definition}\label{d.hexp}
The flow $\varphi_t$ is said to be $h$-expansive at scale $\vep$, if for every $x\in\bM$, it holds
$$
h_{top}(\Gamma_\vep(x),X)=0.
$$ 
\end{definition}

It has been proven in~\cite{LVY} that almost expansivity at scale $\vep>0$ implies $h$-expansivity at the same scale. Furthermore, it is well known from~\cite{B72} that $h$-expansivity at any scale implies that $h_\mu(\phi_1)$ is upper semi-continuous as a function of $\mu$, which in turn implies the existence of equilibrium states for every continuous potential. This is due to the following famous result of Bowen~\cite{B72}:
\begin{proposition}\label{p.stablize1}
Assume that $\varphi_t$ is $h$-expansive at scale $\vep>0$. Then for any finite measure partition $\cA$ with $\diam \cA := \max_{A\in\cA}\diam A<\vep$ where the diameter is measured by the $d_t$-metric with $t=1$, it holds that $h_\mu(\varphi_1,\cA) = h_\mu(\varphi_1)$. 
\end{proposition}

\begin{remark}\label{r.Pexp}
In~\cite{CT16} the authors define the {\em pressure of obstructions to expansivity at scale $\vep$} as
$$
P_{\exp}^\perp (\phi,\vep) = \sup_{\mu\in\cM_X^e(\bM)}\left\{h_\mu(\varphi_1)+\int \phi\,d\mu: \mu(\operatorname{NE}(\vep))>0\right\};
$$
here $\cM_X^e(\bM)$ is the collection of ergodic measures in $\cM_X(\bM)$, and $\operatorname{NE}(\vep) = \bM\setminus \Exp_\vep(X)$.
Following the standard notation of $\sup\emptyset = -\infty$, we see that if $\varphi_t$ is almost expansive at scale $\vep$ then $P_{\exp}^\perp (\phi,\vep') = -\infty$ for all $\vep'\le\vep$.
\end{remark}

We also need the following proposition which is folklore:
\begin{proposition}\label{p.stablize2}
If $\varphi_t$ is almost expansive at scale $\vep>0$, then for every $\gamma\in[0,\vep/2]$ it holds
$$
P(\phi,\gamma) = P(\phi).
$$
\end{proposition}
This proposition can be easily proven using~\cite[Proposition 3.7]{CT16} and Remark~\ref{r.Pexp}.

\subsection{Decomposition of orbit segments}\label{s.2.3}
The main observation in~\cite{CT16} is that the uniqueness of equilibrium states can be obtained if the collection of ``bad orbit segments'' has small topological pressure comparing to the rest of the system.  For this purpose, they define:
\begin{definition}\cite[Definition 2.3]{CT16}
A decomposition  $(\cP,\cG,\cS)$ for $\cD\subset \bM\times \RR^+$ consists of three collections $\cP,\cG,\cS\subset \bM\times \RR^+$ and three functions $p,g,s:\cD\to \RR^+$ such that for every $(x,t)\in \cD$, the values $p=p(x,t), g=g(x,t)$ and $s=s(x,t)$ satisfy $t=p+g+s$, and
\begin{equation}\label{e.decomp}
(x,p)\in\cP,\hspace{0.5cm} (\varphi_px,g)\in\cG,\hspace{0.5cm} (\varphi_{p+g}x,s)\in\cS.
\end{equation}
Given a decomposition $(\cP,\cG,\cS)$ and real number $M\ge0$, we write $\cG^M$ for the set of orbit segments $(x,t)\in \cD$ with $p< M$ and $s< M$.
\end{definition}
We assume that $\bM\times \{0\}$ (whose elements are identified with empty sets) belongs to $\cP\cap\cG\cap\cS$. This allows us to decompose orbit segments in trivial ways. Following~\cite[(2.9)]{CT16}, for $\cC\in\bM\times\RR^+$ we define the slightly larger collection $[\cC]\supset \cC$ to be
\begin{equation}\label{e.[C]}
[\cC]:= \{(x,n)\in \bM\times \mathbb{N}:(\varphi_{-s}x,n+s+t) \in\cC \mbox{ for some } s,t\in[0,1)\}.
\end{equation}
This allows us to pass from continuous time to discrete time.

\subsection{Bowen property}\label{s.2.4}
The Bowen property, also known as the bounded distortion property, was first introduced by Bowen in~\cite{B75} for maps and by Franco~\cite{Franco} for flows.
\begin{definition}\label{d.Bowen}
Given $\cC\subset  \bM\times \RR^+$, a potential $\phi$ is said to have the Bowen property on $\cC$ at scale $\vep>0$, if there exists $K>0$ such that
\begin{equation}\label{e.Bowen}
\sup\left\{|\Phi_0(x,t) - \Phi_0(y,t)|:(x,t)\in\cC, y\in B_t(x,\vep)\right\}\le K.
\end{equation}
\end{definition}
The constant $K$ is sometimes called the {\em distortion constant}. Note that H\"older potentials for uniformly hyperbolic systems (in the case of flows, this precludes the existence of singularities) have Bowen property on $\bM\times \RR^+$. Also note that if $\phi$ has the Bowen property on $\cG$ at scale $\vep$ with distortion constant $K$, then $\phi$ has the Bowen property on $\cG^M$ at the  same scale for every $M>0$, with distortion constant $K(M) = K+2M\Var(\phi,\vep)$.

\subsection{Specification}\label{s.2.5}
The specification property plays a central role in the work of Bowen~\cite{B75} and Climenhaga-Thompson~\cite{CT16}. Roughly speaking, it states that `good' orbit segments (on which there is large pressure and the Bowen property) can be shadowed by regular orbits, with bounded transition time from one segment to the next.
\begin{definition}
We say that $\cG\subset\bM\times\RR^+$ has weak specification at scale $\delta$ if there exists $\tau>0$ such that for every finite orbit collection $\{(x_i,t_i)\}_{i=1}^k\subset \cG$, there exists a point $y$ and a sequence of ``gluing times'' $\tau_1,\ldots,\tau_{k-1}$ with $\tau_i\le \tau$ such that for $s_j = \sum_{i=1}^{j}t_i+\sum_{i=1}^{j-1}\tau_i$ and $s_0=\tau_0=0$, we have 
\begin{equation}\label{e.spec}
d_{t_j}(\varphi_{s_{j-1}+\tau_{j-1}}(y), x_j)<\delta \mbox{ for every } 1\le j\le k.
\end{equation}
The constant $\tau = \tau(\delta)$ is referred to as the {\em maximum gap size}.  
\end{definition}
As in~\cite{CT16} we will sometimes say that $\cG$ has (W)-specification, or simply specification. This version of the specification is weak in the sense that the ``transition times'' $\{\tau_i\}$ are only assumed to be bounded by, rather than equal to, $\tau$. 

\begin{definition}\label{d.tailspec}
We say that $\cG$ has {\em tail (W)-specification}  at scale $\delta$ if there exists $T_0>0$ such that $\cG\cap(\bM\times [T_0,\infty))$ has (W)-specification at scale $\delta$. We may also say that $\cG$ has (W)-specification at scale $\delta$ for $t>T_0$ if we need to declare the choice of $T_0$.
\end{definition}

\subsection{The main result of~\cite{CT16}}
The main theorem of~\cite{CT16} gives the existence and uniqueness of equilibrium states for systems with Bowen property, specification and a ``pressure gap'' on certain ``good orbits''.
\begin{theorem}\cite[Theorem 2.9]{CT16}\label{t.CT}
Let $\varphi_t$ be a continuous  flow on a compact metric space $\bM$, and $\phi:\bM\to\RR$ a continuous potential function. Suppose that there are $\vep>0,\delta>0$ with $\vep>40\delta$ such that $\varphi_t$ is almost expansive at scale $\vep$,\footnote{It is worth noting that the version we cited here is slightly weak than\cite[Theorem 2.9]{CT16} since we replaced the condition  $P^\perp_{\exp}(\phi,\vep)<P(\phi)$ by the assumption that $\varphi_t$ is almost expansive at scale $\vep$ (see Remark~\ref{r.Pexp}). However, this is done only for simplicity's sake, and we believe that Theorem~\ref{m.1} can be proven under the original assumption. } and $\cD\subset\bM\times\RR^+$ which admits a decomposition $(\cP,\cG,\cS)$ with the following properties:
\begin{enumerate}[label={(\Roman*)}]
\item[(I')] For every $M>0$, $\cG^M$ has tail (W)-specification at scale $\delta$;
\item[(II)]  $\phi$ has the Bowen property at scale $\vep$ on $\cG$;
\item[(III)] $P(\cD^c\cup [\cP]\cup[\cS], \phi,\delta,\vep)<P(\phi)$.
\end{enumerate}
Then there exists a unique equilibrium state for the potential $\phi$.
\end{theorem}
Comparing to Theorem~\ref{m.1}, it is clear that assumptions (II) and (III) are the same; (I'), on the other hand, is stronger than (I).

The main feature of this theorem is that it only requires knowledge of the system at a {\em fixed scale}.  That said, it is still difficult to obtain the tail specification of $\cG^M$ at scale $\delta>0$ for every $M>0$, as required by Assumption (I'). In most of the applications, this is done using the following lemma.

\begin{lemma}\label{l.tailspec}\cite[Lemma 2.10]{CT16}
Suppose that $\cG$ has tail specification at all scales $\delta>0$, then so does $\cG^M$ for every $M>0$. 
\end{lemma}

To summarize, the Bowen property on $\cG$ and the pressure gap in (III) only require  a fixed scale of $\vep>40\delta$, but in order to verify (I) one usually has to establish specification of $\cG$ at {\bf any scale}. Although this can be proven in most existing applications, it turns out to be the main obstruction if one hopes to apply their result to flows with singularities.  This is the primary motivation of the current paper.

\section{Choice of the parameters}\label{s.2'}

Recall that $\vep$ is the scale in which the Bowen property holds on $\cG$, and $\delta$ is the scale of the specification property on $\cG$. 
Unlike~\cite{CT16}, in this paper we will not attempt to choose the (near) optimal ratio of  $\vep/\delta = 40$  which only reveals itself at the end of~\cite[Section 4]{CT16}. This is due to the use of $\cG^1$ (see Lemma~\ref{l.key}) which changes the scale of the specification property by a multiple of $L_X$.  Instead, we take, throughout this article,
\begin{equation}\label{e.parameter}
\vep = 1000L_X\delta, \,\,\,\,\rho = 11L_X\delta, \,\,\,\,\mbox{ and }\,\,\rho' = 10L_X\delta.
\end{equation} 
The choice of $\rho$ and $\rho'$ will only become relevant in Section~\ref{s.5} and onward.  Let us mention that $\rho'$ is the size of the separated sets for which an equilibrium state $\mu$ is constructed (Section~\ref{s.5.1}), and $\rho$ is the scale on which the Gibbs property holds for orbit segments in $\cG^1$ (Lemma~\ref{l.gibbs}).

Finally, for each $\delta>0$ we define the closed interval
\begin{equation}\label{e.I}
I_\delta = [4L_X\delta,\,\, 100L_X\delta].
\end{equation}
Throughout this article, the parameter $\gamma$ will be taken from $I_\delta$. Two choices of $\gamma$, namely $\gamma = 2\rho = 22L_X\delta$ and $\gamma = \rho'/2 = 5L_X\delta$, are particularly important in later sections.

The constant $M$ played an important role in~\cite{CT16}.  As mentioned before, the goal of this paper is to remove $M$ from Assumption (I') of Theorem~\ref{t.CT}. To this end, in Section~\ref{s.5} we will first apply~\cite[Lemma 4.8]{CT16} (see Lemma~\ref{l.4.8} below) with $\alpha_1 = 1, \alpha_2 = 1/2$ and $\gamma = 2\rho=22L_X\delta$; this results in a fixed constant
\begin{equation}\label{e.M}
M_0 = M(2\rho, 1)
\end{equation} 
where $M(\cdot,\cdot)$ is given by~\eqref{e.MT} in Lemma~\ref{l.4.8}.

Later in Section~\ref{s.6}, we have to consider other choices of $M$. To be more precise:
\begin{itemize}
	\item in the proof of Proposition~\ref{p.abscts} we apply Lemma~\ref{l.4.8} again with $\gamma = 2\rho$, $\alpha_1 = C_{\frac12}$, $\alpha_2=\frac12$ to get   
	$$
	\overline M = M(2\rho, C_{\frac12});
	$$
	here $C_{(\cdot)}$ is given by Lemma~\ref{l.6.1}; note that $C_{(\cdot)}$ depends implicitly on $M_0$, which has been fixed throughout; 
	
	\item in Proposition~\ref{p.ergodic} we apply Lemma~\ref{l.4.8} with $\gamma = 2\rho$, $\alpha_1 = C_{\frac12 a}$, $\alpha_2=\frac12$ to get   
	$$
	\widehat M = M(2\rho, C_{\frac12 a});
	$$
	here $a>0$ is a small constant given by~\eqref{e.a}, and  $C_{(\cdot)}$ is given by Lemma~\ref{l.6.1}.
\end{itemize}


\section{Lower and upper bounds on the partition function without specification on $\cG^M$}\label{s.3}

In this section, we collect some useful estimations from~\cite{CT16} that use (I) the specification on $\cG$ at scale $\delta$, (II) the Bowen property on $\cG$ at scale $\vep$,  and (III) the pressure gap, but do not involve (I') the specification  on $\cG^M$ for some (or every) $M>0$. As a result, all the lemmas in this section can be directly applied to our setting without any modification. We will not include the proof of these results here, but direct the interested readers to~\cite[Section 3 and 4]{CT16}. 


\begin{lemma}\label{l.4.1}\cite[Lemma 4.1]{CT16} For every $\gamma>0$ and $t_1,\ldots, t_k>0$ we have 
$$
\Lambda(\bM, 2\gamma,t_1+\cdots t_k)\le \prod_{j=1}^k \Lambda(\bM,\gamma,\gamma,t_j).
$$
\end{lemma}

This immediately leads to the following lower bound on $\bM$:
\begin{lemma}\label{l.4.2}\cite[Lemma 4.2]{CT16} Assume that $\varphi_t$ is almost expansive at scale $\vep=1000L_X\delta$. For every $t>0$ and $\gamma\in I_\delta$, we have 
$$
\Lambda(\bM, \gamma,\gamma,t)\ge e^{tP(\phi)}.
$$
\end{lemma}

The next lemma requires the tail specification property on $\cG$ at a fixed scale. Also recall that if the potential $\phi$ has the Bowen property on $\cG$ at scale $\vep$, then it has the Bowen property on $\cG$ at scale $\zeta>0$ for every $\zeta\le \vep$.
\begin{lemma}\cite[Proposition 4.3]{CT16}\label{l.4.3}
Suppose that $\cG$ has specification at scale $\delta$ for $t>T_0$ with maximum gap size $\tau$, and the potential $\phi$ has the Bowen property on $\cG$ at scale $\vep$. Then for every $\gamma\in I_\delta$, there is a constant $C_1>0$ so that for every $k\in\NN$, $t_1,\ldots, t_k\ge T_0$, with $T:=\sum_{i=1}^kt_i + (k-1)\tau$ and for any $\theta<\gamma/2-\delta$, we have 
$$
\prod_{j=1}^{k} \Lambda(\cG,\gamma,0,t_j)\le C_1^k\Lambda(\bM,\theta,0,T).
$$
\end{lemma}

An immediate corollary is the following lemma, which deals with the two-scale partition function. 

\begin{lemma}\cite[Corollary 4.6]{CT16}\label{l.4.6}
Under the assumptions of Lemma~\ref{l.4.3}, for every $\eta_1 \le\vep$ and for every $\eta_2>0$ one has 
$$
\prod_{j=1}^{k}\Lambda(\cG,\gamma,\eta_1,t_j)\le (e^KC_1)^k\Lambda(\bM, \theta,\eta_2,T).
$$
Here $K$ is the constant in the Bowen property.
\end{lemma}

Taking logarithm and sending $k$ to infinity, we obtain the upper bound on $\cG$.

\begin{proposition}\cite[Proposition 4.7]{CT16}\label{p.4.7}
Suppose that $\cG$ has tail specification at scale $\delta>0$ , and $\phi$ has the Bowen property on $\cG$ at scale $\vep$. Then for every $\gamma\in I_\delta$ and $\eta\in[0,\vep]$, there exists constant $C_2>0$ for which 
$$
\Lambda(\cG, \gamma,\eta,t)\le C_2 e^{tP(\phi)}
$$
for any $t\ge 0$.
\end{proposition}

\begin{remark}\label{r.4.5}
We remark that the choice of $C_2$ in the previous proposition can be made independent of $\gamma\in I_\delta$ and $\eta$, due to the monotonicity of $\Lambda(\cC,\delta,\vep,t)$ in both scales (see~\eqref{e.monotone}). In particular, one only needs to take $C_2$ given by  $\gamma = 4L_X\delta =\min I_\delta$ and $\eta = \vep$.
\end{remark}

The next crucial step is to obtain a lower bound of $\Lambda(\cG,\gamma,\eta,t)$ of the form $C_3e^{tP(\phi)}$. More importantly, one would like to prove, if possible, that for every orbits collection $\cC\subset\bM\times\RR^+$ with ``large pressure'', $\cC\cap\cG$ must have ``large pressure'' as well. However this cannot be easily done, even with the help of the pressure gap assumption (III). Instead, the authors of~\cite{CT16} introduced $\cG^M$, and proved that $\cC\cap \cG^M$, for $M$ sufficiently large, must have large pressure; furthermore, the choice of $M$ depends on the pressure gap assumption (III). This is the main reason that the specification on $\cG^M$ is needed for every $M>0$.

This discussion is summarized in the next lemma. The version that we cited here is indeed weaker than~\cite[Lemma 4.8]{CT16} (in terms of the scale for the Bowen property and for the pressure gap; see (3) below) but still suffices for our use. 
Note that despite $\cG^M$ being introduced, the specification is only assumed to hold on $\cG$.

\begin{lemma}\cite[Lemma 4.8]{CT16}\label{l.4.8}
Let $(\cP,\cG,\cS)$ be a decomposition for $\cD\subset \bM\times\RR^+$ such that
\begin{enumerate}
\item $\cG$ has (tail) specification at scale $\delta$;
\item the potential $\phi$ has the Bowen property on $\cG$ at scale $\vep$; and
\item $P(\cD^c\cup [\cP]\cup[\cS], \phi,\delta,\vep)<P(\phi)$.
\end{enumerate}
Then for every $\alpha_1,\alpha_2>0$ and $\gamma\in I_\delta$, there exists $M=M(\gamma,\alpha_1,\alpha_2)\in\RR^+$ and $T_1=T_1(\gamma,\alpha_1,\alpha_2)\in\RR^+$ such that the following holds:
\begin{itemize}
\item for any $t\ge T_1$ and $\cC\subset \bM\times \RR^+$ with $\Lambda(\cC,2\gamma,2\gamma,t)\ge\alpha_1e^{tP(\phi)}$, we have
\begin{equation}\label{e.4.8}
\Lambda(\cC\cap\cG^M,2\gamma,2\gamma,t)\ge(1-\alpha_2)\Lambda(\cC,2\gamma,2\gamma,t)\ge c_1e^{tP(\phi)},
\end{equation}
where $c_1 = \alpha_1(1-\alpha_2)$.

\end{itemize}
\end{lemma}

\begin{remark}\label{r.4.8}
Throughout this article we always take 
$\alpha_2 = \frac12$. Consequently we will write 
\begin{equation}\label{e.MT}
M(\gamma, \alpha_1)=M(\gamma, \alpha_1,\frac12) \mbox{ and } T_1 (\gamma, \alpha_1)= T_1(\gamma, \alpha_1,\frac12)
\end{equation} 
to highlight their dependence on $\gamma$ and $\alpha_1$. In this and the next section, we will, for the most part, take $\alpha_1=1$. However, other values of $\alpha_1$ will become relevant in Section~\ref{s.6} (recall the discussion at the end of Section~\ref{s.2'}).
\end{remark}


The next two lemmas are proven in~\cite{CT16}. For the convenience of our readers, we also weaken the scale of the Bowen property in order to be consistent with Section~\ref{s.2'}. 
\begin{lemma}\label{l.4.9}\cite[Lemma 4.9]{CT16}
Assume that $\cG$ has tail specification at scale $\delta$, $\varphi_t$ is almost expansive at scale $\vep$, and $\phi$ has the Bowen property  on $\cG$ at scale $\vep$. Then for every $\gamma \in I_\delta$ and $\alpha_2>0$, the constants $M\in\NN$ and $T_1\in\RR^+$ given by Lemma~\ref{l.4.8} with $\alpha_1=1$ satisfies that for every $t\ge T_1$,
$$
\Lambda(\cG^M, 2\gamma,2\gamma,t)\ge (1-\alpha_2)\Lambda(\bM,2\gamma,2\gamma,t)\ge (1-\alpha_2)e^{tP(\phi)}.
$$
\end{lemma}

\begin{lemma}\label{l.4.10}\cite[Proposition 4.10]{CT16}
Let $\vep,\delta,\gamma$ be as in Lemma~\ref{l.4.9}, and $M,T_1$ be given by Lemma~\ref{l.4.8} with $\alpha_1=1$ and $\alpha_2=\frac12$. Then there exists $L_1 = L_1(\gamma)\in\RR^+$ such that for $t\ge T_1$, 
$$
\Lambda(\cG^M,2\gamma,t)\ge e^{-L_1}e^{tP(\phi)}.
$$
As a consequence, $\Lambda(\bM,2\gamma,t)\ge e^{-L_1}e^{tP(\phi)} $ for $t\ge T_1$.
\end{lemma}

We remark that both lemmas are direct consequences of Lemma~\ref{l.4.8}, where $M,T_1$  are obtained by taking $\alpha_1=1$ and $\cC = \bM\times\RR^+$. To obtain Lemma~\ref{l.4.10} one takes $\alpha_2 = \frac12$ and uses the Bowen property on $\cG^M$ at scale $2\gamma<\vep$ with distortion constant  $K(M) = K+2M\Var(\phi,\vep)$. As a result, $L_1$ depends implicitly on $M$ and consequently on $\gamma\in I_\delta$.




\section{Removing $M$ - lower bound on $\cG^1$}
In this section we will show the transition from $\cG^M$ to $\cG^1$. The constant $M$ in this Section should be considered as a function of $\gamma\in I_\delta$ and $\alpha_1>0$ (we only take $\alpha_2=\frac12$). More specific choices of $M$ will be made in the next two sections. Also recall that $\cG^1\subset \cD$ is the collection of orbit segments such that $p<1$ and $s<1$. 

We first state the following result which improves Proposition~\ref{p.4.7}. Recall the definition of $I_\delta$ from~\eqref{e.I}.

\begin{lemma}\label{l.G1}
Suppose that $\cG$ has tail specification at scale $\delta>0$, and $\phi$ has the Bowen property on $\cG$ at scale $\vep$. Then there exists a constant $C''_2>0$ such that for every $\gamma\in I_\delta$ and $\eta\in[0,\vep]$, it holds
$$
\Lambda(\cG^1, \gamma,\eta,t)\le C''_2 e^{tP(\phi)}
$$
for any $t\ge 0$.
\end{lemma}

\begin{proof}
If 
$\cG$ has tail specification at scale $\delta$, then $\cG^1$ has tail specification at scale $L_X\delta$ with the same shadowing orbits. Also note that if $\phi$ has the Bowen property on $\cG$ at scale $\vep$, then it also has the Bowen property on $\cG^1$ at the same scale, but with a larger distortion constant $K(1) = K+2\Var(\phi,\vep)$. 

Now the result follows from the proof of  Proposition~\ref{p.4.7} (and the previous lemmas that lead to it) with $\cG$ replaced by $\cG^1$. The uniformity of $C_2''$ in $\gamma\in I_\delta$ and $\eta\in[0,\vep]$ is due to Remark~\ref{r.4.5}.
\end{proof}

Given $\cC\subset\bM\times\RR^+$ and $i,j\in\NN$, we write
\begin{equation}\label{e.fC}
\varphi_{i,j}(\cC) = \{(\varphi_i(x),t-(i+j)):(x,t)\in\cC\}.
\end{equation}
This is the orbit collection formed by removing the starting and ending segments from the orbit segments in $\cC$.

For a positive real number $r>0$ we denote by $[r]$ its integer part. The next key lemma allows us to obtain a lower bound of the partition function on  $\cG^1$ at the price of a time change $t\to T(t)$ which differs from $t$ by at most $2M$. Note that the change of scale from $(2\gamma,2\gamma)$ to $(\gamma,3\gamma)$ is the main reason why we need  uniform estimates of certain constants for $\gamma\in I_\delta$.

\begin{lemma}\label{l.key}
Let $(\cP,\cG,\cS)$ be a decomposition for $\cD\subset \bM\times\RR^+$ such that
\begin{enumerate}
\item $\cG$ has (tail) specification at scale $\delta$;
\item the potential $\phi$ has the Bowen property on $\cG$ at scale $\vep$; and
\item $P(\cD^c\cup [\cP]\cup[\cS], \phi,\delta,\vep)<P(\phi)$.
\end{enumerate}
For every $\alpha_1>0$ and $\gamma\in I_\delta$, let $M = M(\gamma,\alpha_1), T_1 = T_1(\gamma,\alpha_1) $ be the constants given by~\eqref{e.MT}, i.e.,  from Lemma~\ref{l.4.8} with $\alpha_2=\frac12$. Then there exists $L = L(\gamma,\alpha_1)$, such that the following statement holds:
\begin{itemize}
\item 
 for every $t\ge T_1$ and $\cC\subset\bM\times\RR^+$ 
with $\Lambda(\cC,2\gamma,2\gamma,t)\ge\alpha_1e^{tP(\phi)}$, there exists $t_p,t_s\in[0,[M]]\cap \NN$, such that, with $T(t) = t-(t_p+t_s)$,
\begin{equation}\label{e.4.2}
\Lambda\left(\varphi_{t_p,t_s}(\cC)\cap \cG^1,{\gamma,3\gamma},T(t)\right)\ge L \cdot \Lambda(\cC, 2\gamma,2\gamma,t),
\end{equation}
where $\varphi_{t_p,t_s}(\cC)$ is defined by~\eqref{e.fC}.
\end{itemize}
\end{lemma}

\begin{proof}

For every $\gamma\in I_\delta$ and $\alpha_1>0$, we can take $\alpha_2 = \frac12$ in Lemma~\ref{l.4.8} to obtain constants $M = M(\gamma,\alpha_1), T_1 = T_1(\gamma,\alpha_1) $ such that for every $\cC\subset\bM\times\RR^+$ with  $\Lambda(\cC,2\gamma,2\gamma,t)\ge\alpha_1e^{tP(\phi)}$, we have
\begin{equation}\label{e.4.2.0}
\Lambda(\cC\cap\cG^M,2\gamma,2\gamma,t)\ge\frac12\Lambda(\cC,2\gamma,2\gamma,t)
\end{equation}
whenever $t>T_1$.

For any $\beta_1>0$ we take $E_t\subset \cC\cap \cG^M$  a $(t,2\gamma)$-separated set with
\begin{equation}\label{e.4.2.3}
\sum_{x\in E_t}e^{\Phi_{2\gamma}(x,t)}>\Lambda(\cC\cap\cG^M,2\gamma,2\gamma,t)-\beta_1.
\end{equation}
Given $x\in E_t$ with $[p(x,t)] = i$ and $[s(x,t)]=k$, we have 
$$
x\in [\cP]_i,\hspace{0.2in} \varphi_i(x)\in\cG^1_{t-(i+k)},\hspace{0.2in} \varphi_{t-k}(x)\in[\cS]_k,
$$
where $[\cP]$ and $\cS$ are defined as in~\eqref{e.[C]}. Note that $\max\{i,k\}<M$ since $E_t\subset \cG^M_t$. Given $i,k = 1,\ldots,[M]$, we define
$$
E_t(i,k) = \{x\in E_t: [p(x,t)] = i, [s(x,t)] = k\}.
$$

Now for each $i=1,\ldots,[M]$ we take $E^P_i\subset [\cP]_i $ any $(i,\gamma)$-separated set; choose $E^S_k\subset [\cS]_k$ similarly. We also take $E^G_s\subset \left(\varphi_{i,k}(\cC)\cap \cG^1\right)_s$, $s = t-(i+k)$, to be any $(s,\gamma)$-separated set.\footnote{Note that by Lemma~\ref{p.4.7}, $ \left(\varphi_{i,k}(\cC)\cap \cG^1\right)_s\ne\emptyset$ for some $i,k\in [0,[M]]$. For those $i,k$ for which the set $ \left(\varphi_{i,k}(\cC)\cap \cG^1\right)_s$ is indeed empty, the corresponding $E_{t-(i+k)}^G$ is taken to be empty, and the summation $\sum_{y\in E_{t-(i+k)}^G} e^{\Phi_{3\gamma}(y,t-(i+k))}$ on the next page is set to be zero.}

Then we define an injection
$$
\pi:E_t(i,k)\to E^P_i\times E^G_{t-(i+k)} \times E^S_k, \,\,\pi(x) = (x_1,x_2,x_3)
$$
in the following way:
\begin{itemize}
\item $x_1\in E^P_i$ is such that $x\in \overline{B}_i(x_1,\gamma)$;
\item  $x_2\in E^G_{t-(i+k)}$ is such that $\varphi_i(x)\in\overline{B}_{t-(i+k)}(x_2,\gamma)$;
\item  $x_3\in E^S_k$ is such that $\varphi_{t-k}(x)\in \overline{B}_k(x_3,\gamma)$.
\end{itemize}
The injectivity is obvious: if $\pi(x)=\pi(y) = (x_1,x_2,x_3)$, then we have $$d_i(x,y) \le d_i(x,x_1)+d_i(x_1,y) \le 2\gamma;$$ the same holds for $d_{t-(i+k)}(\varphi_i(x), \varphi_i(y))$ and  $d_{k}(\varphi_{t-k}(x), \varphi_{t-k}(y))$. This shows that $d_t(x,y)\le 2\gamma$. Since $E_t(i.k)$ is $(t,2\gamma)$-separated, we must have $x=y$.

Writing $\pi_1,\pi_2$ and $\pi_3$ for the components of $\pi$, we note that for $x\in E_t(i,k)$ it holds
$$
\Phi_{2\gamma}(x,t)\le \Phi_{3\gamma}(\pi_1 x, i)+\Phi_{3\gamma}(\pi_2x,t-(i+k)) + \Phi_{3\gamma}(\pi_3x, k).
$$
Then the left-hand side of~\eqref{e.4.2.3} can be estimated as
\begin{align*}
&\sum_{x\in E_t}e^{\Phi_{2\gamma}(x,t)}\\
 = &\sum_{i,k\in[0,[M]]\cap\NN}\sum_{x\in E_t(i,k)} e^{\Phi_{2\gamma}(x,t)}\\ 
\le &\sum_{i,k\in[0,[M]]\cap\NN}\sum_{x\in E_t(i,k)}e^{\Phi_{3\gamma}(\pi_1 x, i)}\cdot e^{\Phi_{3\gamma}(\pi_2x,t-(i+k))} \cdot e^{\Phi_{3\gamma}(\pi_3x, k)}\\
\le &\sum_{i,k\in[0,[M]]\cap\NN} \sum_{y_1\in E_i^P}e^{\Phi_{3\gamma}(y_1, i)}\cdot \sum_{y_2\in E_{t-(i+k)}^G}e^{\Phi_{3\gamma}(y_2, t-(i+k))}\cdot \sum_{y_3\in E_k^S}e^{\Phi_{3\gamma}(y_3, k)}\\
\le &\sum_{i,k\in[0,[M]]\cap\NN}\Lambda([\cP],\gamma,3\gamma,i)\cdot \Lambda(\varphi_{i,k}(\cC)\cap \cG^1,\gamma,3\gamma,t-(i+k))\cdot\Lambda([\cS],\gamma,3\gamma,k).     \\
\end{align*}
Here we slightly abuse  notation in the case of either $i=0$ or $k=0$ (or both) by setting  the corresponding term $e^{\Phi_{3\gamma}(\pi_1 x, i)}$ or $e^{\Phi_{3\gamma}(\pi_3 x, k)}$ to be one.

Since $M$ is independent of $\cC$, we can take $c_2>1$ depending on $M$ but not on $\cC$ (and hence not on the choice of $E_t$), such that 
$$
\max_{i,j\in[0,[M]]\cap \NN}\Big\{\Lambda([\cP],\gamma,3\gamma,i),\,\, \Lambda([\cS],\gamma,3\gamma,k)\Big\}<c_2.
$$
This gives
\begin{align*}
&\sum_{x\in E_t}e^{\Phi_{2\gamma}(x,t)}\le (c_2)^2\cdot \sum_{i,k\in[0,[M]]\cap\NN} \Lambda\left(\varphi_{i,k}(\cC)\cap \cG^1,\gamma,3\gamma,t-(i+k)\right).
\end{align*} 
As a result, there exist $i,k\in[0,[M]]\cap \NN$ such that 
\begin{align*}
 &\Lambda\left(\varphi_{i,k}(\cC)\cap \cG^1,\gamma,3\gamma,t-(i+k)\right)\\
 \ge &(c_2([M]+1))^{-2}\cdot \sum_{x\in E_t}e^{\Phi_{2\gamma}(x,t)}\\
\numberthis \label{e.4.2.4} \ge &(c_2([M]+1))^{-2}(\Lambda(\cC\cap\cG^M,2\gamma,2\gamma,t)-\beta_1)
\end{align*} 
where the last inequality follows from the choice of $E_t$ by~\eqref{e.4.2.3}.

To conclude, we take a sequence $(\beta^n_1)_{n=1}^\infty$ with $\beta_1^n\to 0$ as $n\to\infty$. Since there are only finitely many choices of $(i,k)$, there exists a subsequence $(n_j)_j$ along which the choices of $(i,k)$ are the same. Fix one such choices and take $t_p = i$, $t_s = k$ and set $T(t)=t-(t_p+t_s)$. \eqref{e.4.2.0} and~\eqref{e.4.2.4} then gives
\begin{align*}
\Lambda\left(\varphi_{t_p,t_s}(\cC)\cap \cG^1,\gamma,3\gamma,T(t)\right)\ge &(c_2([M]+1))^{-2}\cdot \Lambda(\cC\cap\cG^M,2\gamma,2\gamma,t)\\
\ge &(c_2([M]+1))^{-2}\frac12\Lambda(\cC,2\gamma,2\gamma,t).
\end{align*}
Letting $L =\frac12 (c_2([M]+1))^{-2}$, we recover~\eqref{e.4.2} and conclude the proof of the lemma.

\end{proof}

\begin{remark}
Note that the conclusion of Lemma~\ref{l.4.8} implies that $\cG^M_t$ is non-empty for all $t$ large enough. However the same cannot be said about $\cG^1$ (or $\cG$). Lemma~\ref{l.key} only shows that there are infinitely many values of $t$, each of which is less than $2M$ from the next, such that $\cG^1_t$ is non-empty. A similar conclusion holds for $\cG$. 
\end{remark}

\begin{remark}
Since $\phi$ has the Bowen property on $\cG^1$ at scale $\vep>3\gamma$ for every $\gamma\in I_\delta$, the second scale $3\gamma$ on the left-hand side of~\eqref{e.4.2} is superfluous if one replaces $L$ by $Le^{-K(1)}$ where  $K(1) = K+2\Var(\phi,\vep)$. . 
\end{remark}

Next we present some consequences of Lemma~\ref{l.key}.


\begin{lemma}\label{l.cor1}
Let $\vep,\delta,\gamma$ be as in Lemma~\ref{l.key}, and let $M,T_1$ be given by~\eqref{e.MT} with $\alpha_1=1$. Then there exists a  constant $L_2 = L_2(\gamma)\in\RR^+$ such that for every $t\ge T_1$, there exists $T(t)\in[t-2M,t]$ such that
\begin{equation}\label{e.4.5.1}
\Lambda(\cG^1,\gamma,3\gamma,T(t))\ge e^{-L_2} e^{T(t)P(\phi)}. 
\end{equation}
Furthermore, there exists $L_3\in\RR^+$ for which
\begin{equation}\label{e.4.5.2}
\Lambda(\cG^1,\gamma,T(t))\ge e^{-L_3} e^{T(t)P(\phi)}. 
\end{equation}
\end{lemma}

\begin{proof}
Recall that $\Lambda(\bM,2\gamma,2\gamma,t)\ge e^{tP(\phi)}$ by Lemma~\ref{l.4.2}. Applying Lemma~\ref{l.key} on $\cC = \bM\times\RR^+$ with $\alpha_1=1$, we obtain $M, L\in\RR^+$ such that for every $t\ge T_1$, there exists $t_p,t_s=0,\ldots,[M]$ such that, with $T(t) = t-t_p-t_s$,
\begin{align*}
\Lambda(\cG^1,\gamma,3\gamma,T(t))\ge &\Lambda\left(\varphi_{t_p,t_s}(\bM\times\RR^+)\cap \cG^1,\gamma,3\gamma,T(t)\right)\\
\ge & L \cdot \Lambda(\bM, 2\gamma,2\gamma,t)\\
\ge& Le^{tP(\phi)}\\
\ge &Le^{(t_p+t_s)P(\phi)}e^{T(t)P(\phi)}.
\end{align*}
Since $t_p$ and $t_s$ are bounded by $M$, we can take $L_2>0$ sufficiently large so that $e^{-L_2}\le Le^{(t_p+t_s)P(\phi)}$.  So~\eqref{e.4.5.1} holds.

To obtain~\eqref{e.4.5.2}, recall that if $\phi$ has the Bowen property on $\cG$, then it has the Bowen property on $\cG^1$ at the same scale but with a different distortion constant $K(1) = K+2\Var(\phi,\vep)$. We thus obtain from~\eqref{e.4.5.1}:
$$
\Lambda(\cG^1,\gamma,T(t))\ge e^{-K(1)}\Lambda(\cG^1,\gamma,3\gamma,T(t))\ge e^{-(K(1)+L_2)}e^{T(t)P(\phi)},
$$
which gives~\eqref{e.4.5.2} with $L_3 = L_2+K(1)$.

\end{proof}

We conclude this section with the following upper and lower bound for the partition function $\Lambda(\bM,\gamma,t)$ (this means that the second scale is 0). Note that although the statement resembles that of~\cite[Lemma 4.11]{CT16}, the proof must be modified due to the proof of~\cite[Lemma 4.11]{CT16} relying (implicitly) on the tail specification of $\cG^M$. 

\begin{lemma}\label{l.4.11}
Assume that $\cG$ has tail specification at scale $\delta$,  $\varphi_t$ is almost expansive at scale $\vep$, and $\phi$ has the Bowen property  on $\cG$ at scale $\vep$.  For any $\gamma\in I_\delta$, let $T_1 = T_1(\gamma,1)$ be the constant given in Lemma~\ref{l.key} with $\alpha_1 =1$. Then there exists  $C_3 = C_3(\gamma)>0$ such that for every $t\ge T_1$,
$$
C_3^{-1}e^{tP(\phi)}\le\Lambda(\bM, 2\gamma,t)\le \Lambda(\bM,2\gamma,2\gamma, t)\le C_3 e^{tP(\phi)}.
$$
\end{lemma}
\begin{proof}
By Lemma~\ref{l.4.10} and the monotonicity of $\Lambda(\bM,*,t)$ (see~\eqref{e.monotone}), we have 
$$ e^{-L_1}e^{tP(\phi)} \le \Lambda(\cG^M,2\gamma,t)\le \Lambda(\bM,2\gamma,t ) 
$$
which gives the first inequality. 
The second inequality is obvious. To prove the third inequality, we use Lemma~\ref{l.key} on $\cC =\bM\times\RR^+$ and $\alpha_1=1$ to get
\begin{align*}
\Lambda(\cG^1, \gamma,3\gamma,T(t))\ge &\Lambda\left(\varphi_{t_p,t_s}(\bM\times\RR^+)\cap \cG^1,2\gamma,2\gamma,T(t)\right)\\
\ge & L \cdot \Lambda(\bM, 2\gamma,2\gamma,t).
\end{align*}

On the other hand,  Lemma~\ref{l.G1} with $\gamma\in I_\delta$ and $\eta = 3\gamma$ gives
$$
\Lambda(\cG^1,\gamma,3\gamma,T(t))\le C_2''e^{tP(\phi)},
$$
so we have 
$$
\Lambda(\bM, 2\gamma,2\gamma,t)\le C_2''L^{-1}e^{tP(\phi)}.
$$
\end{proof}

\section{Lower Gibbs bound on $\cG^1$}\label{s.5}
From now on we will assume that the assumptions of Theorem~\ref{m.1} hold. More specifically, we have, for $\vep=1000L_X\delta>0$: 
\begin{enumerate}[label={(\Roman*)}]
\item[{(0)}] $\varphi_t$ is almost expansive at scale $\vep$;
\item  $\cG$ has tail (W)-specification at scale $\delta$;
\item the potential $\phi$ has the Bowen property at scale $\vep$ on $\cG$;
\item $P(\cD^c\cup[\cP]\cup[\cS],\phi,\delta,\vep)<P(\phi)$. 
\end{enumerate}
In particular, all the lemmas in the previous two sections hold. We fix  $\rho = 11L_X\delta$ and $\rho' = 10L_X\delta$ as mentioned in Section~\ref{s.2'}, and note that $\rho,\rho'\in I_\delta$. We will also take
$$
M_0 = M(2\rho,1)
$$
as discussed at the end of  Section~\ref{s.2'}.

\subsection{Construction of an equilibrium state}\label{s.5.1}
The construction of an equilibrium state $\mu$ is standard. For $t>0$ we take a maximizing $(t, \rho')$-separated set $E_t$ for $\Lambda(\bM,\rho',t)$.\footnote{Following our notation in Section~\ref{s.2}, this means the second scale is zero.} Then we consider 
\begin{equation}\label{e.mu}
\begin{split}
&\nu_t := \frac{\sum_{x\in E_t}\exp(\Phi_0(x,t))\cdot \delta_x}{\sum_{x\in E_t}\exp(\Phi_0(x,t))}, \mbox{ and }\\
&\mu_t := \frac1t \int_0^t(\varphi_s)_*\nu_t\, ds;
\end{split}
\end{equation}
here $\delta_x$ is the point mass at $x$. We take a subsequence $n_k\to\infty $ which we assume to be integers, such that $\mu_{n_k}$ converges to  a measure $\mu$ in the weak-* topology. Since the flow $\varphi_t$ is almost expansive, $\mu$ must be an equilibrium state (\cite[Lemma 4.14, Proposition 4.15]{CT16}; see also~\cite[Theorem 8.6]{Wal}).

\subsection{Gibbs property on $\cG^1$}\label{s.5.2}
Below we will show that $\mu$ has certain form of lower Gibbs property along orbit segments in $\cG^1$. This improves~\cite[Lemma 4.16]{CT16} by dropping $M$. In the next section, we will use Lemma~\ref{l.gibbs} to show that all equilibrium states must be absolutely continuous with respect to $\mu$.

\begin{lemma}\label{l.gibbs}
There exists $T_2>0$, $Q>0$ such that for every $(x,t)\in \cG^1$ with $t>T_2$, we have 
\begin{equation}\label{e.gibbs}
\mu(B_t(x,\rho))\ge Qe^{-tP(\phi)+\Phi_0(x,t)}.
\end{equation}
\end{lemma}

\begin{remark}\label{r.5.1.M}
In~\cite[Lemma 4.16]{CT16} a similar result is obtained for orbit segments in $\cG^M$ for all $M$ sufficiently large, with the constant $Q_M$ depending on $M$. Lemma~\ref{l.gibbs}  generalizes this result to $\cG^1$; however, it is worth noting that the constant $Q$ here also depends implicitly on the constant $M_0 = M(2\rho,1)$.
\end{remark}

The proof of this lemma is similar to the proof of~\cite[Lemma 4.16]{CT16} with three major differences:
\begin{itemize}
\item In~\cite[Lemma 4.16]{CT16} the separated sets $E'_u$ are taken from $\cG^M_u$; here we take the separated sets from $\cG^1_u$ and have to change the scale of the specification from $\delta$ to $L_X\delta$.
\item We use the lower bound on $\cG^1$ obtained  in Lemma~\ref{l.cor1}, which requires us to make a small change of time $u' = T(u)$ with $|u'-u|\le 2M_0$. This affects the choice of certain constants, making them depend on $M_0$.
\item Lemma~\ref{l.key}, and consequently Lemma~\ref{l.cor1}, changes the scale of the partition function from $(2\gamma,2\gamma)$ to $(\gamma,3\gamma)$. Therefore we need to apply Lemma~\ref{l.cor1} at scale $\gamma=2\rho$ as oppose to  $\gamma=\rho$ in~\cite{CT16}.
\end{itemize}

Now we are ready to prove Lemma~\ref{l.gibbs}.
\begin{proof}
First, recall that since $\varphi_t$ has tail (W)-specification on $\cG$ with scale $\delta$, it also  has tail (W)-specification on $\cG^1$ at scale $L_X \delta$ with gluing time $\tau^1$. 

We apply Lemma~\ref{l.cor1} with $\gamma = 2\rho = 22L_X\delta\in I_\delta$ to obtain constants $M_0 = M(2\rho,1), T_1 = T_1(2\rho,1)$ and $L_3 = L_3(2\rho)$ so that whenever $t>T_1$,
$$
\Lambda(\cG^1,{2\rho},T(t))\ge e^{-L_3} e^{T(t)P(\phi)}
$$ 
for some $T(t)\in [t-2M_0,t]$. Without loss of generality we assume that the specification property on $\cG^1$ holds for all orbit segments with time longer than $T_1$. Given $(x,t)\in\cG^1$ with $t>T_1$, we estimate $\mu(B_t(x,\rho))$ by estimating $\nu_s(\varphi_{-r}(B_t(x,\rho)))$
for $s\gg t$ and $r\in(\tau^1+T_1, s-2\tau^1-2T_1-t)$. 


Given $s$ and $r$, let $u_1:=r-\tau^1$ and
$u_2:=s-r-t-\tau^1$, so $u_1,u_2>T_1$. By Lemma~\ref{l.cor1}, denote by $u_1^\prime=T(u_1)$ and $u_2^\prime=T(u_2)$, and take a $(u_1^\prime,{2\rho})$-separated set of $\cG^1_{u_1^\prime}$ which we denote by  $E^\prime_{u_1^\prime}$, such that
\begin{equation}\label{eq.u1prime}
\sum_{x\in E'_{u_1'}}e^{\Phi_0(x,u_1')} \geq \frac{1}{2}e^{-L_3}e^{u_1^\prime P(\phi)};
\end{equation}
$E^\prime_{u_2^\prime}$ is taken as a $(u_2^\prime, 2\rho)$-separated set of $\cG^1_{u_2^\prime}$ with a similar property. See~\cite[Fig 4]{CT16}.

We use  ideas similar to the proof of~\cite[Proposition 4.3]{CT16} to construct a map $\pi: E_{u_1^\prime}^\prime \times E_{u_2^\prime}^\prime \to E_s$
as follows. By the specification property of $\cG^1$, for each $\textbf{x} = (x_1,x_2)\in E_{u_1^\prime}^\prime\times E_{u_2^\prime}^\prime$,
we can find a point $y(\textbf{x})\in\bM$  and $\tau_1(\textbf{x}), \tau_2(\textbf{x})\in[0,\tau^1]$ so that
\begin{align*}
y(\textbf{x})\in&\, B_{u_1^\prime}(x_1,L_X\delta),\\
\numberthis \label{e.5.1.1}\varphi_{u_1^\prime+\tau_1(\textbf{x})}(y(\textbf{x}))\in & \,B_t(x,L_X\delta),\\
\varphi_{u_1^\prime+\tau_1(\textbf{x})+t+\tau_2(\textbf{x})}(y(\textbf{x}))\in& \,B_{u_2^\prime}(x_2,L_X\delta).
\end{align*}

Recall that $\rho^\prime=10L_X\delta$ and $\rho = 11L_X\delta$. 
Let $E_s$ denote the maximizing $(s,\rho^\prime)$-separated set of $\bM$ used in the construction of $\nu_s$ and $\mu_s$. 
Let $\pi: E^\prime_{u'_1}\times E^\prime_{u'_2}\to E_s$ be given by
choosing a point $\pi(\textbf{x})\in E_s$ such that
$$d_s(\pi(\textbf{x}),\varphi_{(u_1^\prime - u_1)+ (\tau_1(\textbf{x})-\tau^1)}(y(\textbf{x})))\leq \rho^\prime.$$

For any $\textbf{x}\in E^\prime_{u_1}\times E^\prime_{u_2}$, we have
\begin{align*}
& d_t(\varphi_r(\pi(\textbf{x})),x)\\\leq&\, d_t\left(\varphi_r(\pi(\textbf{x})),\varphi_{r+(u_1^\prime-u_1)+(\tau_1(\textbf{x})-\tau^1)}(y(\textbf{x}))\right)+d_t\left(\varphi_{r+(u_1^\prime-u_1)+(\tau_1(\textbf{x})-\tau^1)}(y(\textbf{x})),x\right)\\<&\,\rho^\prime+L_X\delta=\rho,
\end{align*}
where the second inequality follows from the choice of $y(\textbf{x})$ by~\eqref{e.5.1.1} and the observation that (recall $u_1 = r-\tau^1$)
$$
r+(u_1'-u_1) + (\tau_1(\textbf{x})-\tau^1) = u_1 + u_1'-u_1 +\tau_1(\textbf{x}) = u_1'+\tau_1(\textbf{x}).
$$
This shows that
\begin{equation}\label{eq.pointapproach}
\pi(\textbf{x})\in \varphi_{-r}(B_t(x,\rho)).
\end{equation}

The proof of~\cite[Lemma 4.4]{CT16} shows there is a constant $D>0$ such that 
\begin{equation}\label{e.6.2.1}
\#\pi^{-1}(z)\leq D^3
\end{equation}
for every
$z\in E_s$. Moreover, a mild adaptation of the proof of~\cite[Lemma 4.5]{CT16}\footnote{Note that~\cite[Lemma 4.4 and 4.5]{CT16} are under the assumption that $\cG$ has specification at scale $\delta$, and therefore can be adapted to our setting with minimal modification.} gives the existence of $C_4$ (depending on $M_0$) such
that
\begin{equation}\label{eq.seperatepoints}
\Phi_0(\pi(\textbf{x}),s)\geq -C_4+\Phi_0(x_1,u_1^\prime)+\Phi_0(x_2,u_2^\prime)+\Phi_0(x,t).
\end{equation}
Note that the dependence on $M_0$ is due to $u_1'+u_2'+t\in[s-4M_0-2\tau^1,s]$ whereas in~\cite[Lemma 4.16]{CT16} one has $u_1+u_2+t\in[s-2\tau^1,s]$.

The rest of the proof is largely the same as~\cite[Lemma 4.16]{CT16}. We have 
\begin{align*}
\numberthis\label{e.5.1.2}\nu_s(\varphi_{-r}(B_t(x,\rho)))=&\frac{\sum_{z\in E_s} e^{\Phi_0(z,s) \delta_z(\varphi_{-r}(B_t(x,\rho)))}}{\sum_{z\in E_s} e^{\Phi_0(z,s)}}
\\
\geq &\,\, D^{-3}\left(\sum_{\textbf{x}\in E_{u_1^\prime}^\prime\times E_{u_2^\prime}^\prime} e^{\Phi_0(\pi(\textbf{x}),s)}\right)\cdot\left(\sum_{z\in E_s} e^{\Phi_0(z,s)}\right)^{-1}.
\end{align*}
where we use \eqref{eq.pointapproach} and~\eqref{e.6.2.1} for the estimate on the numerator.

To control the last term, we apply Lemma~\ref{l.4.11} with $2\gamma = \rho'$, i.e., $\gamma = 5L_X\delta\in I_\delta$ to obtain, for $s$ sufficiently large,
$$
\sum_{z\in E_s} e^{\Phi_0(z,s)}\le \Lambda(\bM,\rho',s) \le C_3e^{sP(\phi)}.
$$
This, together with~\eqref{e.5.1.2}, gives
\begin{equation}\label{e.6.1.2}
\nu_s(\varphi_{-r}(B_t(x,\rho)))\ge D^{-3} C_3^{-1}e^{-sP(\phi)}\cdot \sum_{\textbf{x}\in E_{u_1^\prime}^\prime\times E_{u_2^\prime}^\prime} e^{\Phi_0(\pi(\textbf{x}),s)}.
\end{equation}

On the other hand, by \eqref{eq.seperatepoints} and \eqref{eq.u1prime} we have
\begin{equation*}
\begin{split}
\sum_{\textbf{x}\in E_{u_1^\prime}^\prime \times E_{u_1^\prime}^\prime} e^{\Phi_0(\pi(\textbf{x},s))}&\geq e^{-C_4} e^{\Phi_0(x,t)}\left(\sum_{x_1\in E_{u_1^\prime}^\prime} e^{\Phi_0(x_1,u_1^\prime)}\right)\left(\sum_{x_2\in E_{u_2^\prime}^\prime} e^{\Phi_0(x_2,u_2^\prime)}\right)\\
&\geq \frac{1}{4}e^{-C_4}e^{\Phi_0(x,t)}e^{-2L_3}e^{u_1^\prime P(\phi)}e^{u_2^\prime P(\phi)}.\\
\end{split}
\end{equation*}
Together with \eqref{e.6.1.2} and the fact that $s=u_1+u_2+t+2\tau^1$, this gives
$$\nu_s(\varphi_{-r}(B_t(x,\rho)))\geq C_5 e^{\Phi_0(x,t)}e^{(u_1^\prime+u_2^\prime-s)P(\phi)}= C_5 e^{(-2\tau^1)P(\phi)}e^{-tP(\phi)+\Phi_0(x,t)},$$
for every $r\in (T_1+\tau^1, s-2\tau^1-2T_1-t)$. Integrating over $r$ gives
\begin{equation*}
\begin{split}
\mu_s(B_t(x,\rho))&\geq \frac{1}{s}\int_{T_1+\tau^1}^{s-2\tau^1-2T_1-t} \nu_s(\varphi_{-r}(B_t(x,\rho)))\\
&\geq \left(1-\frac{t+3\tau^1+3T_1}{s}\right)C_5e^{-2\tau^1 P(\phi)}e^{-tP(\phi)+\Phi_0(x,t)}.\\
\end{split}
\end{equation*}
We conclude the proof by sending $s\to\infty.$

\end{proof}

The next lemma improves~\cite[Lemma 4.17]{CT16} and will be useful to establish the ergodicity of $\mu$.
\begin{lemma}\label{l.mixing}
There exists $Q'>0$ such that for every $(x_1,t_1)$, $(x_2,t_2)\in \cG^1$ with $t_1,t_2\ge T_1,$ and every $q>2\tau^1+T_1$, there exists $q'\in[q-2\tau^1-2M_0,q]$ such that 
$$
\mu(B_{t_1}(x_1,\rho)\cap \varphi_{-(t_1+q')} B_{t_2}(x_2,\rho))\ge Q'e^{-(t_1+t_2)P(\phi) + \Phi_0(x_1,t_1)+\Phi_0(x_2,t_2)}.
$$
Furthermore, we can choose $N\in\NN$ such that $q'$ can be taken such that $q' = q-k-\frac{2i}{N}\tau^1$ for some $k \in \{0,1,\ldots,2[M_0]\}$ and $i\in\{0,\ldots,N\}$.
\end{lemma}
\begin{proof}
The proof is similar to the proof of~\cite[Lemma 4.17]{CT16} (which itself follows the same idea as~\cite[Lemma 4.16]{CT16}) and is thus omitted. 
We make the same modifications as in the proof of Lemma~\ref{l.gibbs}, using $\cG^1$ instead of $\cG^M$ when taking the $(u',{2\rho})$-separated sets $E'_{u'}$, and using Lemma~\ref{l.cor1} to obtain lower bounds on the partition sum over said $E'_{u'}$. The constant $Q'$ must inevitably depend on $M_0$. The lower bound of the range of $q'$ is changed to $q-2\tau^1-2M_0$ from $q-2\tau(M)$ in~\cite[Lemma 4.17]{CT16} because of  the time change $u\to u' = T(u)$ which differs from $u$ by an integer no more than $2M_0$.
\end{proof}


\section{Proof of Theorem~\ref{m.1}}\label{s.6}
In this section we will provide the proof of Theorem~\ref{m.1}. The proof is similar to~\cite{CT16} and consists of two steps: (1) every equilibrium state $\nu$ cannot be mutually singular with $\mu$, the equilibrium state constructed in the previous section, and (2) $\mu$ is ergodic. The proof of (1) relies on the observation that each equilibrium state is related to a orbit collection $\cC$ with large pressure (\cite[Lemma 4.18]{CT16}, see Lemma~\ref{l.6.1} below), and the fact that $\mu$ has the lower Gibbs property (Lemma~\ref{l.gibbs}). The proof of (2) uses Lemma~\ref{l.mixing} which can be seen as a form of mixing with respect to the Bowen balls.

Below we have to make several crucial modifications comparing to~\cite[Section 4]{CT16}: in Lemma~\ref{l.separate} below where we have to approximate the typical points of an equilibrium state $\nu$ with preimages of Bowen balls, whereas in~\cite{CT16} it is done using adapted partitions. This together with the lower Gibbs bound (Lemma~\ref{l.gibbs}) shows that all other equilibriums states must be absolutely continuous with respect to $\mu$. Finally we prove the ergodicity of  $\mu$  in Section~\ref{s.6.3}, which gives the uniqueness.

\subsection{Adapted partitions}\label{s.6.1}
First we recall the concept of the adapted partitions constructed by Bowen~\cite{B74}. 

\begin{definition}\label{d.adapted}
Let $\gamma>0$ and $E_t$ be a maximizing $(t,\gamma)$-separated set. A measurable partition $\cA_t$ of $\bM$ is called {\em adapted to $E_t$}, if for every $w\in \cA_t$  there is $x\in E_t$ with 
\begin{equation}\label{e.adapted}
B_t(x,\gamma/2)\subset w \subset \overline{B}_t(x,\gamma).
\end{equation}
\end{definition}
On the other hand, given each $x\in E_t$ there is a unique $w\in\cA_t$ for which~\eqref{e.adapted} holds. We denote it by $w_x$ to emphasis the dependence on $x$.

The following lemma is proven in~\cite{CT16}, which states that for every equilibrium state $\nu$, a positive $\nu$-measure set must have large pressure.
\begin{lemma}(see \cite[Lemma 4.18]{CT16})\label{l.6.1}
Let $\vep,\delta$ be as before and let $\gamma\in I_\delta$. For every $\alpha_3\in(0,1)$, there exists a constant $C_{\alpha_3}>0$ with the following property: let $\nu$ be any equilibrium state for the potential $\phi$, and let $\{E_t\}_{t>0}$ be a family of maximizing $(t,2\gamma)$-separated sets for $\Lambda(\bM,2\gamma,t)$ with adapted partitions $\cA_t$. Then for every $t>0$, if $E_t'\subset E_t$ satisfies $\nu\left(\bigcup_{x\in E_t'}w_x\right)\ge\alpha_3,$ then letting $\cC = \{(x,t):x\in E_t'\}$, we have
$$
\Lambda(\cC,2\gamma,2\gamma, t)\ge C_{\alpha_3} e^{tP(\phi)}.
$$  
\end{lemma}
The proof is omitted. We remark that the proof of this result in~\cite{CT16} only uses the assumption that $\nu$ is almost expansive at scale $\vep$ and the upper bound of  $\Lambda(\bM,2\gamma,2\gamma,t)$ given by Lemma~\ref{l.4.11}, and therefore can be applied to our setting.

\subsection{No mutually singular equilibrium states}\label{s.6.2}
Let us recall that $\mu$ is the equilibrium state constructed in Section~\ref{s.5.1}.  The goal of this section is to prove the following result.
\begin{proposition}\label{p.abscts}
Assume that the assumptions of Theorem~\ref{m.1} hold. Then there is no equilibrium state $\nu$ that is mutually singular with $\mu$.
\end{proposition}
The proof of this proposition requires the following approximation lemma, which improves~\cite[Proposition 3.10]{CT16} by taking into account the change of time  $t\to T(t)$.

\begin{lemma}\label{l.separate}
Assume that $\varphi_t$ is almost expansive at scale $\vep.$ Let $\nu_1,\nu_2$ be two invariant probability measures that are mutually singular to each other. Then for $\xi<\vep$ and every $\beta>0$, there exist compact sets $Q_t$ for all $t$ sufficiently large, with $\nu_1(Q_t)\ge 1-\beta$ and $\nu_2(Q_t) = 0$, such that for every $M>0$ and $0\le i,j \le M$, we have 
$$
\limsup_{t\to\infty}\,\,\nu_2\left(\bigcup_{y\in Q_t}\varphi_{-i}\left( B_{t-(i+j)}(\varphi_i(y),\xi)\right)\right)\le \beta.
$$
\end{lemma}
\begin{proof}
For any set $w\in \bM$ and $s\ge 0$, let
$$
\diam_{[-s,s]}w = \sup_{x,y\in w}\inf\left\{d(\varphi_{t_1}(x), \varphi_{t_2}(y)): t_1,t_2\in[-s,s]\right\},
$$
and note that if $s'>s$ then $\diam_{[-s',s']}w\le \diam_{[-s,s]}w $. 
Since $\varphi_t$ is  almost expansive at scale $\vep$, it follows that $\Gamma_\vep(x)$, the two-sided infinite Bowen ball at $x$ with scale $\vep$, belongs to a orbit segment $\varphi_{[-s,s]}(x)$ for $\nu_1,\nu_2$ almost every $x$ and some $s=s(x)>0$. Writing 
$$
B_{[-s,s]}(x,\xi) = \{y\in\bM: d_{2s}(\varphi_{-s}(y),\varphi_{-s}(x))< \xi\}
$$
for the two-sided Bowen ball at $x$, we then have, for $\nu_1,\nu_2$ almost every $x$,
\begin{equation}\label{e.6.3.a}
\diam_{[-s(x),s(x)]}B_{[-t,t]} (x,\xi)\xrightarrow{t\to\infty} 0 
\end{equation}
for some $s(x)>0$.

Since $\nu_1,\nu_2$ are mutually singular, there exists disjoint invariant sets $Q^1,Q^2$ with $\nu_i(Q^j)=\delta_{ij}$ where $\delta_{ij}$ is the Kronecker delta. By Egorov's theorem we can take disjoint compact sets $K_1$ and $K_2$ with $K_i\subset Q^i$, such that
\begin{itemize}
\item $\nu_1(K_1)>1-\beta/2$ and $\nu_2(K_1)=0$;
\item $\nu_2(K_2)>1-\beta/2$ and $\nu_1(K_2)=0$;
\item there exists $s_0>0$ such that on $K_1$ we have $s(x)<s_0$, and the convergence in~\eqref{e.6.3.a} is uniform.
\end{itemize}

Write $K_i^{s_0} = \bigcup_{t\in[-s_0,s_0]}\varphi_t(K_i), i=1,2$ and note that $K_1^{s_0},K_2^{s_0}$ are compact and disjoint since they each belongs to $Q^i$. Consequently, there exists $a>0$ such that $K_1^{s_0},K_2^{s_0}$ are separated by $a$; that is,
\begin{equation}\label{e.aa}
d(\varphi_{t_1}(x), \varphi_{t_2}(y))\ge a, \mbox{ for all }x\in K_1,y\in K_2, t_1,t_2\in[-s_0,s_0].
\end{equation}

Since the convergence of~\eqref{e.6.3.a} is uniform on $K_1$, we obtain $t_0>0$ such that for any $t\ge t_0$ and $x\in K_1$,
$$
\diam_{[-s_0,s_0]}B_{[-t,t]}(x,\xi)< a/2.
$$
Thus, the set $K_1':=\bigcup_{x\in K_1} B_{[-t_0,t_0]}(x,\xi)$ is disjoint with $K_2$, since any point of intersection $y\in K_2\cap B_{[-t_0,t_0]}(x,\xi)$ with  $x\in K_1$ will clearly violate~\eqref{e.aa}.
Consequently, we have $\nu_2(K_1') <\beta$.

For any $t>0$, let $Q_t=\varphi_{-t/2}(K_1)$. Below we will prove that $Q_t$ satisfies the desired property.

 Because both $\nu_1$ and $\nu_2$ are
invariant probabilities, the first two properties of $Q_t$ are clearly satisfied, that is, $\nu_1(Q_t)\geq 1-\beta$ and $\nu_2(Q_t)=0$. Next, observe that for any $M\geq 0$, $t>2(t_0+M)$ and $0\le i,j\le M$, we have for any $y\in Q_t$
\begin{equation}\label{e.6.3.b}
\varphi_{t/2-i}(B_{t-(i+j)}(\varphi_i(y),\xi))\subset B_{[-t_0,t_0]}(\varphi_{t/2}(y),\xi) \subset K_1'.
\end{equation}

Writing 
$$
Q_t' = \bigcup_{y\in Q_t}\varphi_{-i}\left( B_{t-(i+j)}(\varphi_i(y),\xi)\right), 
$$
by~\eqref{e.6.3.b} and the invariance of $\nu_2$, we have
$$
\nu_2(Q_t') = \nu_2(\varphi_{t/2} (Q_t'))\le \nu_2 \left(\bigcup_{y\in Q_t}\varphi_{t/2-i}(B_{t-(i+j)}(\varphi_i(y),\xi))\right)\le \nu_2(K_1') <\beta.
$$
The proof is complete.
\end{proof}

Now we are ready to prove Proposition~\ref{p.abscts}.
\begin{proof}
We prove by contradiction. Let $\nu$ be any equilibrium state that is mutually singular with $\mu$.   Recall from Section~\ref{s.2'} the prescribed parameters: $\vep = 1000L_X\delta, \rho = 11L_X\delta$.

Let $\beta\in(0,\frac12)$ be an arbitrary constant whose choice will be specified later. We apply Lemma~\ref{l.separate} with $\nu_1= \nu$, $\nu_2 = \mu$,
 $\xi =2\rho$ and $\beta$ to obtained the compact sets $Q_t$ for all $t$ sufficiently large. 

Let $E_t$ be a maximizing $(t,4\rho)$-separated set of $\bM$, and $\cA_t$ an adapted partition. Define
$$
U_t : = \bigcup_{\substack{w\in\cA_t\\w\cap Q_t\ne\emptyset}} w.
$$
Then we have $\nu(U_t)\ge \nu(Q_t) \ge 1-\beta> \frac12$ for all $t$ sufficiently large. This allows us to apply Lemma~\ref{l.6.1} with $\gamma=2\rho$ and $\alpha_3=\frac12$ to obtain $C_\frac12$ for which 
$$
\Lambda(\cC,4\rho,4\rho, t)\ge C_\frac12 e^{tP(\phi)};
$$
here $\cC = \{(x,t):x\in E_t\cap U_t\} $. Then for each  $t$, $\cC_t = \{x:(x,t)\in\cC\}$ is a finite set. Also note that $C_\frac12$ depends on $\delta,\rho,\vep$ but not on $\beta$. 

Now we are in a position to apply Lemma~\ref{l.key} again with  $\gamma=2\rho$ and $\alpha_1 = C_{\frac12}$, which gives $\overline L,\overline M,\overline T_1$ such that for $t\ge \overline T_1$, there exists $t_p,t_s \in [0,[\overline M]]\cap \NN$, such that for $T(t) = t-t_p-t_s,$ 
\begin{equation}\label{e.6.4.1}
\Lambda\left(\varphi_{t_p,t_s}(\cC)\cap \cG^1,2\rho,6\rho,T(t)\right)\ge \overline L \cdot \Lambda(\cC, 4\rho,4\rho,t)\ge \overline LC_\frac12 e^{tP(\phi)}.
\end{equation}
Recall that our parameters $\rho,\delta,\vep$ have been fixed throughout (see Section~\ref{s.2'}). This means that the constants $\overline L,\overline M,\overline T_1, C_\frac12$ are also fixed and do not depend on $\beta$.

We take $E_t' = \{x\in E_t\cap U_t: (\varphi_{t_p}(x), T(t))\in \cG^1\}$, 
and note that 
$$\left(\varphi_{t_p,t_s}(\cC)\cap \cG^1\right)_{T(t)} = \{\varphi_{t_p}(x):x\in E_t'\}$$
is a finite set; however, it may not be $(T(t),2\rho)$-separated (despite $E_t$ being $(t,4\rho)$-separated). To deal with this, we take a maximizing $(T(t),2\rho)$-separated subset of $\left(\varphi_{t_p,t_s}(\cC)\cap \cG^1\right)_{T(t)}$, which we denote by $\tilde E'_t$, and define
$$
E_t'' = \varphi_{-t_p}(\tilde E'_t)\subset E_t'.
$$

Now \eqref{e.6.4.1} means that 
$$
\sum_{x\in E_t''} \exp\left(\Phi_{6\rho}(\varphi_{t_p}(x), T(t))\right)\ge \overline LC_\frac12 e^{tP(\phi)}.
$$
By the Bowen property on $\cG^1$ at scale $\vep > 6\rho$, we obtain
$$
\sum_{x\in E_t''} \exp\left(\Phi_{0}(\varphi_{t_p}(x), T(t))\right)\ge \overline LC_\frac12 e^{-K(1)}\cdot e^{tP(\phi)}.
$$

Now we put 
$$
V_t'' =\bigcup_{y\in E_t''} B_{T(t)}(\varphi_{t_p}(y),\rho),
$$
and remark that it is a disjoint union due to $\varphi_{t_p}(E_t'')$ being $(T(t),2\rho)$-separated\footnote{This explains the choice of $\gamma = 2\rho$ earlier.}. By the lower Gibbs property on $\cG^1$ in Lemma~\ref{l.gibbs} at scale $\rho$, it follows that
\begin{align*}
\mu(V_t'') =& \sum_{y\in E^\prime_t} \mu \left(B_{T(t)}(\varphi_{t_p}(y),\rho)\right)\\
\ge & \, Qe^{-T(t)P(\phi)}\sum_{x\in E_t'} \exp\left(\Phi_{0}(\varphi_{t_p}(x), T(t))\right)\\
\numberthis\label{e.6.4.2}\ge & \, Q\overline LC_\frac12e^{-K(1)} \cdot e^{(t-T(t))P(\phi)}.
\end{align*}

Since $t-T(t) \in [0, 2\overline M]$ we see that $\liminf\limits_t\mu(V_t'')>0$ and is independent of $\beta$. 

On the other hand, since $E_t''\subset E_t\cap U_t$, by Lemma~\ref{l.6.1} we have 
$$
\mu(V_t'')=\mu(\varphi_{-t_p}(V_t'')) \le \mu\left(\bigcup_{y\in Q_t}\varphi_{-t_p}\left( B_{T(t)}(\varphi_{t_p}(y),2\rho)\right) \right),
$$
which gives 
\begin{equation}\label{e.6.4.3}
\limsup_t\mu(V_t'')\le \beta
\end{equation}
by Lemma~\ref{l.separate}. Since $\beta$ is arbitrary and is independent of all the constants involved in~\eqref{e.6.4.2}, we can take 
$$
\beta< \min\left\{ Q\overline LC_\frac12e^{-K(1)} \cdot e^{kP(\phi)}:k=0,1,\ldots,2([\overline M]+1)\right\},
$$
causing~\eqref{e.6.4.3} to contradict with~\eqref{e.6.4.2}. This finishes the proof of Proposition~\ref{p.abscts}.
\end{proof}

\subsection{Ergodicity of $\mu$}\label{s.6.3}
To prove Theorem~\ref{m.1}, it only remains to show that $\mu$ is ergodic. 

\begin{proposition}\label{p.ergodic}
The equilibrium state $\mu$ constructed in Section~\ref{s.5.1} is ergodic.
\end{proposition}

\begin{proof}
The proof is similar to~\cite[Proposition 4.19]{CT16}, with Lemma~\ref{l.mixing} replacing~\cite[Lemma 4.17]{CT16}.

We take two measurable invariant sets $P,P'$ with positive $\mu$ measure. We will prove that 
$$
\mu(P\cap P')>0,
$$
which immediately leads to the ergodicity of $\mu$.

We prove this statement by contradiction. Assume that $P,P'$ are invariant, have positive $\mu$ measures and $\mu(P\cap P')=0$.  We then consider the conditional probability measures $\mu_{P}$ and $\mu_{P'}$, and note that they are invariant and mutually singular. 

Let $\beta\in(0,\frac12)$ be an arbitrary constant which will be determined later, we apply Lemma~\ref{l.separate} twice (with $\xi = 2\rho$ and $\beta$ as in the previous proposition) to obtain compact sets $Q_t^P$ and $Q_t^{P'}$ for all $t$ large enough, with the following properties:
\begin{itemize}
\item $\mu_P(Q_t^P)>1-\beta$ and for every $M>0$ and $0\le i,j \le M$, 
$$
\limsup_{t\to\infty}\,\,\mu_{P'}\left(\bigcup_{y\in Q_t^P}\varphi_{-i}\left( B_{t-(i+j)}(\varphi_i(y),2\rho)\right)\right)\le \beta;
$$
\item  $\mu_{P'}(Q_t^{P'})>1-\beta$ and for every $M>0$ and $0\le i,j \le M$, 
$$
\limsup_{t\to\infty}\,\,\mu_{P}\left(\bigcup_{y\in Q_t^{P'}}\varphi_{-i}\left( B_{t-(i+j)}(\varphi_i(y),2\rho)\right)\right)\le \beta.
$$ 
\end{itemize}
We take $E_t$ a maximizing $(t,4\rho)$-separated set of $\bM$ and $\cA_t$ an adapted partition. Define, as before, the approximation of $P$ and $P'$ with elements of $\cA_t$:
$$
U_t^P : = \bigcup_{\substack{w\in\cA_t\\w\cap Q_t^P\ne\emptyset}} w,\hspace{1cm}U_t^{P'} : = \bigcup_{\substack{w\in\cA_t\\w\cap Q_t^{P'}\ne\emptyset}} w.
$$
Then $\mu_P(U_t^P)\ge \frac12$ and $\mu_{P'}(U_t^{P'})\ge \frac12$. Writing 
\begin{equation}\label{e.a}
a = \min \{\mu(P), \mu(P')\},
\end{equation}
we see that $\min\left\{\mu(U_t^P), \mu(U_t^{P'})\right\}\ge\frac12 a$. 
Applying Lemma~\ref{l.gibbs} with $\gamma = 2\rho$ and $\alpha_3 = \frac12a$, we obtain $C_{\frac12a}$, such that 
\begin{equation}\label{e.6.5.1}
\Lambda(\cC^P,4\rho,4\rho, t)\ge C_{\frac12a} e^{tP(\phi)}, \mbox{ and }\Lambda(\cC^{P'},4\rho,4\rho, t)\ge C_{\frac12a} e^{tP(\phi)},
\end{equation}
where $\cC^P = \{(x,t):x\in E_t\cap U_t^P\}$, and $\cC^{P'} $ is defined similarly.

Now we  apply Lemma~\ref{l.key} with $\gamma = 2\rho$ and $\alpha_1 = C_{\frac12a}$, which gives $\widehat L,\widehat M,\widehat T_1$ depending on $\rho$ and $C_{\frac12a}$ but not on $\beta$, such that for $t\ge\widehat T_1$, there exists integers $t_p,t_s,t_p',t_s'\in [0,[\widehat M]]\cap \NN$, such that for $T^P(t) = t-t_p-t_s$ and $T^{P'}(t) = t-t_p'-t_s'$,
\begin{equation}\label{e.6.5.2}
\Lambda\left(\varphi_{t_p,t_s}(\cC^P)\cap \cG^1,2\rho,6\rho,T^P(t)\right)\ge \widehat L  \Lambda(\cC^P, 4\rho,4\rho,t)\ge \widehat LC_{\frac12a} e^{tP(\phi)},\mbox{ and }
\end{equation}
\begin{equation}\label{e.6.5.3}
\Lambda\left(\varphi_{t_p',t_s'}(\cC^{P'})\cap \cG^1,2\rho,6\rho,T^{P'}(t)\right)\ge \widehat L  \Lambda(\cC^{P'}, 4\rho,4\rho,t)\ge \widehat LC_{\frac12a} e^{tP(\phi)}.
\end{equation}

We take $E_t^P = \{x\in E_t\cap U_t^P: (\varphi_{t_p}(x), T^P(t))\in\cG^1\}$ and $E_t^{P'} = \{x\in E_t\cap U_t^{P'}: (\varphi_{t'_p}(x), T^{P'}(t))\in\cG^1\}$ respectively. Take $\tilde E_t^P$ a maximizing $(T^P(t),2\rho)$-separated subset of $\varphi_{t_p}(E_t^P)$, and let 
$$
\overline E_t^P = \varphi_{-t_p}(\tilde E_t^P)\subset E_t^P.
$$
We also define $\overline E_t^{P'} $ in a similar way.
Then by~\eqref{e.6.5.2} and~\eqref{e.6.5.3} and the Bowen property on $\cG^1$, we have
\begin{equation}\label{e.6.5.4}
\begin{split}
\sum_{x\in \overline E_t^P} \exp\left(\Phi_{0}(\varphi_{t_p}(x), T^P(t))\right)\ge \widehat LC_{\frac12a}e^{-K(1)} e^{tP(\phi)},\\
\sum_{x\in \overline E_t^{P'}} \exp\left(\Phi_{0}(\varphi_{t_p'}(x), T^{P'}(t))\right)\ge \widehat LC_{\frac12a}e^{-K(1)} e^{tP(\phi)}.
\end{split}
\end{equation}

Finally we define, as before,
$$
V^P_t =\bigcup_{y\in \overline E^P_t} B_{T^P(t)}(\varphi_{t_p}(y),\rho),\,\,\,\,\,V^{P'}_t =\bigcup_{y\in \overline E^{P'}_t} B_{T^{P'}(t)}(\varphi_{t_p'}(y),\rho),
$$ 
and note that both are disjoint unions due to $\varphi_{t_p}(\overline E^P_t) = \tilde E_t^P$, $\varphi_{t'_p}(\overline E^{P'}_t) = \tilde E_t^{P'}$   being $(T^P(t),2\rho)$ and $(T^{P'}(t),2\rho)$-separated, respectively. 

To simplify notation,  for $x\in \overline E_t^P$ and $y\in \overline E_T^{P'}$ we write $x' = \varphi_{t_p}(x)$, $t_1 = T^P(t)$, $y' = \varphi_{t_p'}(y)$, $t_2 = T^{P'}(t)$.
Then $(x',t_1),(y',t_2)\in\cG^1$. Fix any $q>2\tau^1+\widehat T_1$. By Lemma~\ref{l.mixing}, there exists $N>0$ such that  for any $x\in E_t^P$, $y\in E_t^{P'}$, there exists $q' = q-k-\frac{2i}{N}\tau^1$ for some $k \in \{0,1,\ldots,2[\widehat M]\}$ and $i\in\{0,\ldots,N\}$ such that 
\begin{equation}\label{e.6.5.5}
\begin{split}
\mu(B_{t_1}(x',\rho)\cap \varphi_{-(t_1+q')} B_{t_2}(y',\rho))\ge Q'e^{-(t_1+t_2)P(\phi) + \Phi_0(x',t_1)+\Phi_0(y',t_2)}
\end{split}
\end{equation}
for some constant $Q'>0$ independent of $\beta$.

Note that the left-hand side of~\eqref{e.6.5.5} belongs to $\varphi_{t_p}(V_t^P)\cap \varphi_{-(t_1+q')+t_p'} V_t^{P'}$ with $q'$ (but not $t_1$, $t_p$ and $t_p'$) depending on $x$ and $y$. Summing over $x\in E_t^P$ and $y\in E_t^{P'}$, we obtain from~\eqref{e.6.5.5} and~\eqref{e.6.5.4}
\begin{align*}
&\sum_{\substack{k \in \{0,1,\ldots,2[\widehat M]\}\\i\in\{0,\ldots,N\}}} \mu(\varphi_{t_p}(V_t^P)\cap \varphi_{-(t_1+q-k-\frac{2i}{N}\tau^1)+t_p'} V_t^{P'})\\
\ge &\sum_{\substack{k \in \{0,1,\ldots,2[\widehat M]\}\\i\in\{0,\ldots,N\}\\x\in\overline E_t^P,\,y\in\overline E_t^{P'}}} \mu(B_{t_1}(x',\rho)\cap \varphi_{-(t_1+q-k-\frac{2i}{N}\tau^1)} B_{t_2}(y',\rho))\\
\ge &\sum_{x\in\overline E_t^P,\,y\in\overline E_t^{P'}} Q'e^{-(t_1+t_2)P(\phi) + \Phi_0(x',t_1)+\Phi_0(y',t_2)}\\
\ge &\, Q \left(\widehat LC_{\frac12a}e^{-K(1)}\right)^2e^{(2t-(t_1+t_2))P(\phi)}.
\end{align*}
Noting that $t_1 = T^P(t), t_2 = T^{P'}(t)$ only differ from $t$ by at most $2\widehat M$, for every $t>0$ sufficiently large, there exists $q'' \in [q-2\tau^1-2\widehat M,q+\widehat M] $ such that 
\begin{align*}\numberthis\label{e.6.5.6}
 \mu(V_t^P\cap \varphi_{-(t+q'')}V_t^{P'}) \ge&\, \frac{1}{2(\widehat M+1)(N+1)} Q \left(\widehat LC_{\frac12a}e^{-K(1)}\right)^2e^{(2t-(t_1+t_2))P(\phi)}\\
 \ge&\, c>0
\end{align*}
with $c$ being some constant independent of $\beta$.

On the other hand, writing $s = t+q''$, we have
$$
(V^P_t\cap \varphi_{-s} V_t^{P'})\setminus  (P\cap \varphi_{-s} P')\subset (V_t^P\setminus P)\cup \varphi_{-s} (V_t^{P'}\setminus P')
$$
By the construction of $V_t^P$ and $V_t^{P'}$, we have $\mu(V_t^P\setminus P) < 2\beta$ and the same holds for $V_t^{P'}\setminus P'$. This means that 
$$
\mu(P\cap P') = \mu(P\cap \varphi_{-s} P') \ge \mu(V^P_t\cap \varphi_{-s} V_t^{P'}) - 4\beta \ge c-4\beta.
$$
Taking $\beta$ small enough causes $\mu(P\cap P')  > 0$, contradicting with the assumption that $\mu(P\cap P')=0$. The proof of Proposition~\ref{p.ergodic} is now complete.

\end{proof}

\section*{Acknowledgments} The authors are grateful to Climenhaga and Thompson for  their careful reading and helpful comments, which significantly improved the presentation of the current paper.

\end{document}